\documentclass[a4paper,10pt]{article}
\usepackage{stmaryrd}
\usepackage{amsfonts}
\usepackage{bbm}
\usepackage{amscd}
\usepackage{mathrsfs}
\usepackage{latexsym,amssymb,amsmath,amscd,amscd,amsthm,amsxtra}
\usepackage[dvips]{graphicx}
\usepackage[utf8]{inputenc}
\usepackage[T1]{fontenc}
\usepackage{lmodern}
\usepackage{amssymb}
\usepackage[all]{xy}
\usepackage{nicefrac,mathtools,enumitem}
\usepackage[numbers,sort&compress]{natbib}
\usepackage{microtype}
\usepackage{xcolor}
\newtheorem{thm}{Theorem}[section]
\newtheorem{lem}[thm]{Lemma}
\newtheorem{cor}[thm]{Corollary}
\newtheorem{pro}[thm]{Proposition}
\newtheorem{ex}[thm]{Example}
\newtheorem{rmk}[thm]{Remark}
\newtheorem{defi}[thm]{Definition}

\newcommand {\emptycomment}[1]{}

\newcommand{\be }{\begin{equation}}
\newcommand{\ee }{\end{equation}}

\newcommand{\g}{\frkg}

\newcommand{\huaA}{\mathcal{A}}

\newcommand{\huaV}{\mathcal{V}}

\newcommand{\huaO}{\mathcal{O}}

\newcommand{\frkg}{\mathfrak g}

\newcommand{\br}[1]{   [ \cdot,    \cdot  ]_\frkg   }

\newcommand{\dM}{\mathrm{d}}

\newcommand{\Hom}{\mathrm{Hom}}
\newcommand{\Der}{\mathrm{Der}}

\newcommand{\Ad}{\mathrm{Ad}}

\newcommand{\gl}{\mathfrak {gl}}

\newcommand{\idd}{\mathrm{id}}

\newcommand{\ad}{\mathrm{ad}}
\newcommand{\AD}{\mathfrak{ad}}

\begin{document}
\title{
{Purely Hom-Lie bialgebras
\thanks
 {
Research supported by NSFC (11471139) and NSF of Jilin Province (20140520054JH,20170101050JC).
 }
} }
\author{Liqiang Cai$^1$ and Yunhe Sheng$^{2,3}$  \\
$^1$School of Mathematics and Statistics, Henan University,\\
Kaifeng 475000, Henan, China\\
$^2$Department of Mathematics, Jilin University,\\
 Changchun 130012, Jilin, China\\
$^3$Department of Mathematics, Xinyang Normal University,\\ \vspace{2mm} Xinyang 464000, Henan, China
\\\vspace{3mm}
Email: cailq13@mails.jlu.edu.cn,~shengyh@jlu.edu.cn }

\date{}
\footnotetext{{\it{Keyword}:  Hom-Lie algebras,  Manin triples, purely Hom-Lie bialgebras, classical Hom-Yang-Baxter equations }}

\footnotetext{{\it{MSC}}: 16T25, 17B62, 17B99.}

\maketitle
\begin{abstract}
In this paper, first we show that there is a Hom-Lie algebra structure on the set of $(\sigma,\sigma)$-derivations of an associative algebra.  Then we construct the dual representation of a representation of a Hom-Lie algebra. We introduce the notions of a Manin triple for Hom-Lie algebras and a purely Hom-Lie bialgebra. Using the coadjoint representation, we show that there is a one-to-one correspondence between Manin triples for Hom-Lie algebras and purely Hom-Lie bialgebras. Finally, we study coboundary purely Hom-Lie bialgebras and construct solutions of the classical Hom-Yang-Baxter equations in some special Hom-Lie algebras using Hom-$\mathcal O$-operators.
\end{abstract}

\section{Introduction}

In the study of $\sigma$-derivations of an associative algebra, Hartwig, Larsson
and Silvestrov introduced the notion of a Hom-Lie algebra in \cite{HLS}. Some $q$-deformations of the Witt and the Virasoro algebras have the
structure of a Hom-Lie algebra.
 By  Hartwig, Larsson
and Silvestrov's definition, a Hom-Lie algebra $(L,[\cdot,\cdot]_L,\varsigma)$ is a nonassociative algebra $(L,[\cdot,\cdot]_L)$ together with an algebra homomorphism $\varsigma:L\longrightarrow L$ such that $[\cdot,\cdot]_L$ is skew-symmetric and the following equality holds:
\begin{equation}\label{eq:originalHJ}
  [(\idd+\varsigma)(x),[y,z]_L]_L+[(\idd+\varsigma)(y),[z,x]_L]_L+[(\idd+\varsigma)(z),[x,y]_L]_L=0,\quad \forall x,y,z\in L.
\end{equation}
Because of their close relation
to discrete and deformed vector fields and differential calculus
\cite{HLS,LD1,LD2}, Hartwig, Larsson
and Silvestrov's  Hom-Lie algebras were widely studied. Recently,  Elchinger,  Lundengard,  Makhlouf and  Silvestrov extend the result in \cite{HLS} to the case of $(\sigma,\tau)$-derivations, see \cite{ELMS} for more details.

Then in \cite{MS2}, Makhlouf and Silvestrov modified the definition of the Hom-Lie algebra. In Makhlouf and Silvestrov's new definition, a Hom-Lie algebra $(\g,[\cdot,\cdot]_\g,\phi)$ is a nonassociative algebra $(\g,[\cdot,\cdot]_\g)$ together with an algebra homomorphism $\phi:\g\longrightarrow \g$ such that $[\cdot,\cdot]_\g$ is skew-symmetric and the following equality holds:
\begin{equation}\label{eq:originalHJ}
  [\phi(x),[y,z]_\g]_\g+[\phi(y),[z,x]_\g]_\g+[\phi(z),[x,y]_\g]_\g=0,\quad \forall x,y,z\in \g.
\end{equation}
If the underlying nonassociative algebra is a Lie algebra, the above two definitions are equivalent. Some works have been done under this condition (e.g. \cite{Jin1,Jin2}). However, in general the above two definitions are not equivalent: in Hartwig, Larsson
and Silvestrov's Hom-Lie algebra, if we set $\phi=\idd+\varsigma$, then $\phi$ is not an algebra homomorphism. On the other hand, in Hartwig, Larsson
and Silvestrov's original study, the underlying nonassociative algebra is not a Lie algebra. Even though many important works have been done for Makhlouf
and Silvestrov's Hom-Lie algebra (e.g. \cite{AEM,BM,CIP,MS1,MS2,Makhlouf-Hopf,sheng3,Yao1,Yao2}), there are still lack of concrete examples of such an algebraic structure.

In this paper, we give an example of Makhlouf
and Silvestrov's Hom-Lie algebra. We show that on the set of $(\sigma,\sigma)$-derivations of an associative algebra, there is a natural Hom-Lie algebra structure in the sense of Makhlouf
and Silvestrov. In the sequel, we only consider  Makhlouf
and Silvestrov's Hom-Lie algebra.

The other purpose of this paper is to study the bialgebra theory for Hom-Lie algebras. A bialgebra is a pair $(\g,\g^*)$ such that there are same algebraic structures on $\g$ and $\g^*$ and some compatibility conditions are satisfied.
One of the most interesting bialgebra is a Lie bialgebra (\cite{D}), which contains a Lie algebra $(\g,[\cdot,\cdot])$ and Lie coalgebra $(\g,\Delta)$ such that a compatibility condition between the Lie bracket $[\cdot,\cdot]:\wedge^2\g\longrightarrow \g$ and Lie cobracket $\Delta:\g\longrightarrow\wedge^2\g$  is satisfied. One compatibility condition is that $\Delta$ is a 1-cocycle on $\g$ with the coefficients in $\wedge^2\g$. Another equivalent  compatibility condition is that $(\g\oplus\g^*;\g,\g^*)$ is a Manin triple for Lie algebras. There are already two approaches to study Hom-Lie bialgebras. One is given in \cite{Yao3}, where the compatibility condition is that the Hom-Lie cobracket $\Delta$ is a 1-cocycle on the Hom-Lie algebras $\g$ with the coefficients in $\wedge^2\g$. However, different from the case of Lie bialgebras, there is not a natural Hom-Lie algebra structure on $\g\oplus \g^*$ such that $(\g\oplus \g^*;\g,\g^*)$ is a Manin triple for Hom-Lie algebras. To solve this defect, the authors gave another approach in \cite{sheng1}. However, a quite strong condition needs to be added so that there is a coadjoint representation of $\g$ on $\g^*$.

 To give the correct notion of a Hom-Lie bialgebra, we need to have the correct dual representation first. As pointed in \cite{BM}, for a representation $\rho$ of a Hom-Lie algebra $(\g,[\cdot,\cdot]_\g,\phi)$ on $V$ with respect to $\beta\in\gl(V)$, $\rho^*$ is not a representation of $\g$ on $V^*$ anymore. To solve this problem, we introduce $\rho^\star:\g\longrightarrow\gl(V^*)$, which turns to be a representation of $\g$ on $V^*$ with respect to $(\beta^{-1})^*$. We explain that $\rho^\star$ is actually the right dual representation in the Hom-Lie algebra context (See Remark \ref{rmk:dualrep}). This is the main tool in our study. Then we modify the invariant condition of a symmetric bilinear form on a Hom-Lie algebra. In \cite{BM}, a symmetric bilinear form $B$ on a Hom-Lie algebra $(\g,[\cdot,\cdot]_\g,\phi)$ is called invariant if $B(\phi(x),y)=B(x,\phi(y))$. We modify this condition to be $B(\phi(x),\phi(y))=B(x,y)$. See Remark \ref{rmk:B} for the explanation. This is the main ingredient to give a new type of Hom-Lie bialgebras. With these preparations, we give the notions of a Manin triple for Hom-Lie algebras and a purely Hom-Lie bialgebra. We also study triangular purely Hom-Lie bialgebras in details. Our purely Hom-Lie bialgebra enjoys all the properties that a Lie bialgebra has.

The paper is organized as follows. In Section 2, we show that on the set of $(\sigma,\sigma)$-derivations of an associative algebra, there is a natural Hom-Lie algebra structure (Theorem \ref{derhomliealg}) and we also give some examples. In Section 3, associated to a representation  $\rho$ of a Hom-Lie algebra $(\g,[\cdot,\cdot]_\g,\phi)$ on $V$ with respect to  $\beta\in GL(V)$, we construct a representation  $\rho^\star$ of  $\g$ on $V^*$ with respect to  $(\beta^{-1})^*\in GL(V^*)$,
which is called the dual representation (Lemma \ref{lem:dualrep}). In Section 4, first we give the notion of a quadratic Hom-Lie algebra. Then we give the notions of a Manin triple for Hom-Lie algebras and a purely Hom-Lie bialgebra. We show that there is a one-to-one correspondence between Manin triples for Hom-Lie algebras and  purely Hom-Lie bialgebras (Theorem \ref{thm:equivalent}). In Section 5, we study coboundary purely Hom-Lie bialgebras and introduce the notion of the classical Hom-Yang-Baxter equation. We construct solutions of the classical Hom-Yang-Baxter equation in some semidirect product Hom-Lie algebra using Hom-$\mathcal O$-operators (Theorem \ref{thm:HYB}).

\section*{Acknowledgements}
We give our warmest thanks to the referees for very helpful comments that improve the paper.

\section{Hom-Lie algebra structures on the set of $(\sigma,\sigma)$-derivations}

In this section, we construct a Hom-Lie algebra on the set of $(\sigma,\sigma)$-derivations of an associative algebra.

A Hom-Lie algebra $(\g,[\cdot,\cdot]_\g,\phi)$ is called a {\bf regular Hom-Lie algebra} if $\phi$ is
an algebra automorphism. Throughout the paper, all Hom-Lie algebras are regular. Let $V$ be a vector space, for any $\beta\in GL(V)$, define a bracket $[\cdot,\cdot]_\beta$ on $\gl(V)$ by
$$[A,B]_\beta=\beta A\beta^{-1}B\beta^{-1}-\beta B\beta^{-1}A\beta^{-1},\quad \forall A,B\in\gl(V).$$
Then $(\gl(V),[\cdot,\cdot]_\beta,\Ad_\beta)$ is a regular Hom-Lie algebra, where $\Ad_\beta(A)=\beta A\beta^{-1}$ for all $A\in\gl(V)$. This regular Hom-Lie algebra plays important roles in the representation theory of Hom-Lie algebras. See \cite{SX} for more details.

Let $\huaA$ be an associative algebra over a field $\mathbb K$, and $\sigma$ and $\tau$ denote two algebra endomorphisms on $\huaA$. A {\bf $(\sigma,\tau)$-derivation} on $\huaA$ is a linear map $D:\huaA\longrightarrow \huaA$ such that
\begin{equation}
  D(ab)=D(a)\tau(b)+\sigma(a)D(b),\quad \forall~a,b\in\huaA.
\end{equation}
 The set of all $(\sigma,\tau)$-derivations on $\huaA$ is denoted by $\Der_{\sigma,\tau}(\huaA)$. See \cite{ELMS} for examples of $(\sigma,\tau)$-derivations. 
 Let $\sigma$ be an algebra isomorphism of $\huaA$. Define a skew-symmetric bilinear map $[\cdot,\cdot]_\sigma:\wedge^2\Der_{\sigma,\sigma}(\huaA)\longrightarrow \Der_{\sigma,\sigma}(\huaA)$ by
\begin{equation}\label{dd}
  [D_1,D_2]_\sigma=\sigma\circ D_1\circ \sigma^{-1}\circ D_2\circ \sigma^{-1}-\sigma\circ D_2\circ \sigma^{-1}\circ D_1\circ \sigma^{-1},\quad \forall D_1,D_2\in\Der_{\sigma,\sigma}(\huaA).
\end{equation}

The following lemma ensures that the operation $[\cdot,\cdot]_\sigma$ is well-defined.
\begin{lem}
For all $D_1,D_2\in\Der_{\sigma,\sigma}(\huaA)$, we have $[D_1,D_2]_\sigma\in \Der_{\sigma,\sigma}(\huaA).$
\end{lem}
\begin{proof}
For all $a,b\in\huaA$, by $\sigma(ab)=\sigma(a)\sigma(b)$, we have
\begin{eqnarray*}
~[D_1,D_2]_\sigma(ab)
&=&(\sigma\circ D_1\circ \sigma^{-1}\circ D_2\circ \sigma^{-1}-\sigma\circ D_2\circ \sigma^{-1}\circ D_1\circ \sigma^{-1})(ab)\\
&=&\sigma D_1\Big{(}\sigma^{-1} D_2\big{(}\sigma^{-1}(a)\sigma^{-1}(b)\big{)}\Big{)}
-\sigma D_2\Big{(}\sigma^{-1} D_1\big{(}\sigma^{-1}(a)\sigma^{-1}(b)\big{)}\Big{)}\\
&=&\sigma D_1\Big{(}\sigma^{-1} \big{(}D_2(\sigma^{-1}(a))b+aD_2(\sigma^{-1}b)\big{)}\Big{)}\\
&&-\sigma D_2\Big{(}\sigma^{-1} \big{(}D_1(\sigma^{-1}(a))b+aD_1(\sigma^{-1}b)\big{)}\Big{)}\\
&=&\sigma D_1\Big{(}\sigma^{-1}D_2\big{(}\sigma^{-1}(a)\big{)}\sigma^{-1} (b)+\sigma^{-1} (a)\sigma^{-1} D_2(\sigma^{-1}b)\Big{)}\\
&&-\sigma D_2\Big{(}\sigma^{-1}D_1\big{(}\sigma^{-1}(a)\big{)}\sigma^{-1} (b)+\sigma^{-1} (a)\sigma^{-1} D_1(\sigma^{-1}b)\Big{)}\\
&=&\sigma\Big{(}D_1\big{(}\sigma^{-1}D_2(\sigma^{-1}a)\big{)}b+D_2\big{(}\sigma^{-1}a\big{)}D_1\big{(}\sigma^{-1}b\big{)}\\
&&+D_1(\sigma^{-1}a)D_2(\sigma^{-1}b)+aD_1\big{(}\sigma^{-1}D_2(\sigma^{-1}b)\big{)}
\Big{)}\\
&&-\sigma\Big{(}D_2\big{(}\sigma^{-1}D_1(\sigma^{-1}a)\big{)}b+D_1\big{(}\sigma^{-1}a\big{)}D_2\big{(}\sigma^{-1}b\big{)}\\
&&+D_2(\sigma^{-1}a)D_1(\sigma^{-1}b)+aD_2\big{(}\sigma^{-1}D_1(\sigma^{-1}b)\big{)}
\Big{)}\\
&=&\sigma D_1\big{(}\sigma^{-1}D_2(\sigma^{-1}a)\big{)}\sigma(b)+\sigma D_2\big{(}\sigma^{-1}a\big{)}\sigma D_1\big{(}\sigma^{-1}b\big{)}\\
&&+\sigma D_1(\sigma^{-1}a)\sigma D_2(\sigma^{-1}b)+\sigma(a)\sigma D_1\big{(}\sigma^{-1}D_2(\sigma^{-1}b)\big{)}\\
&&-\sigma D_2\big{(}\sigma^{-1}D_1(\sigma^{-1}a)\big{)}\sigma(b)-\sigma D_1\big{(}\sigma^{-1}a\big{)}\sigma D_2\big{(}\sigma^{-1}b\big{)}\\
&&-\sigma D_2(\sigma^{-1}a)\sigma D_1(\sigma^{-1}b)-\sigma(a)\sigma D_2\big{(}\sigma^{-1}D_1(\sigma^{-1}b)\big{)}\\
&=&[D_1,D_2]_\sigma(a)\sigma(b)+\sigma(a)[D_1,D_2]_\sigma(b),
\end{eqnarray*}
which implies that $[D_1,D_2]_\sigma\in \Der_{\sigma,\sigma}(\huaA).$
\end{proof}

Define $\Ad_\sigma:\Der_{\sigma,\sigma}(\huaA)\rightarrow\Der_{\sigma,\sigma}(\huaA)$  by
\begin{equation}
  \Ad_\sigma(D)=\sigma\circ D\circ\sigma^{-1}.
\end{equation}

\begin{thm}\label{derhomliealg}
With the above notations,  $(\Der_{\sigma,\sigma}(\huaA),[\cdot,\cdot]_\sigma,\Ad_\sigma)$ is a Hom-Lie algebra.
\end{thm}
\begin{proof}
For all $D_1,D_2\in\Der_{\sigma,\sigma}(\huaA)$, we have
\begin{eqnarray*}
\Ad_\sigma([D_1,D_2]_\sigma)
&=&\Ad_\sigma(\sigma\circ D_1\circ \sigma^{-1}\circ D_2\circ \sigma^{-1}
-\sigma\circ D_2\circ \sigma^{-1}\circ D_1\circ \sigma^{-1})\\
&=&\sigma^2\circ D_1\circ \sigma^{-1}\circ D_2\circ \sigma^{-2}-\sigma^2\circ D_2\circ \sigma^{-1}\circ D_1\circ \sigma^{-2}\\
&=&\sigma\circ (\sigma\circ D_1\circ\sigma^{-1})\circ \sigma^{-1}\circ(\sigma\circ D_2\circ\sigma^{-1})\circ \sigma^{-1}\\
&&-\sigma\circ (\sigma\circ D_2\circ\sigma^{-1})\circ \sigma^{-1}\circ (\sigma\circ D_1\circ\sigma^{-1})\circ \sigma^{-1}\\
&=&[\Ad_\sigma (D_1),\Ad_\sigma (D_2)]_\sigma,
\end{eqnarray*}
which implies that $\Ad_\sigma$ is an algebra isomorphism.

Furthermore, for all $D_1,D_2,D_3\in\Der_{\sigma,\sigma}(\huaA)$, we have
\begin{eqnarray*}
&&[[D_1,D_2]_\sigma,\Ad_\sigma (D_3)]_\sigma+c.p.\\
&=&[\sigma\circ D_1\circ \sigma^{-1}\circ D_2\circ \sigma^{-1}-\sigma\circ D_2\circ \sigma^{-1}\circ D_1\circ \sigma^{-1},\Ad_\sigma (D_3)]_\sigma+c.p.\\
&=&\sigma\circ (\sigma\circ D_1\circ \sigma^{-1}\circ D_2\circ \sigma^{-1}-\sigma\circ D_2\circ \sigma^{-1}\circ D_1\circ \sigma^{-1})\circ \sigma^{-1}\circ \Ad_\sigma (D_3)\circ \sigma^{-1}\\
&&-\sigma\circ \Ad_\sigma (D_3)\circ \sigma^{-1}\circ(\sigma\circ D_1\circ \sigma^{-1}\circ D_2\circ \sigma^{-1}-\sigma\circ D_2\circ \sigma^{-1}\circ D_1\circ \sigma^{-1})\circ \sigma^{-1}+c.p.\\
&=&\sigma^2\circ D_1\circ \sigma^{-1}\circ D_2\circ \sigma^{-1}\circ D_3\circ\sigma^{-2}
-\sigma^2\circ D_2\circ \sigma^{-1}\circ D_1\circ \sigma^{-1}\circ D_3\circ\sigma^{-2}\\
&&-\sigma^2\circ D_3\circ \sigma^{-1}\circ D_1\circ \sigma^{-1}\circ D_2\circ\sigma^{-2}
+\sigma^2\circ D_3\circ \sigma^{-1}\circ D_2\circ \sigma^{-1}\circ D_1\circ\sigma^{-2}+c.p.\\
&=&0,
\end{eqnarray*}
where c.p. means cyclic permutation.
Therefore,  $(\Der_{\sigma,\sigma}(\huaA),[\cdot,\cdot]_\sigma,\Ad_\sigma)$ is a Hom-Lie algebra.
\end{proof}

\begin{rmk}
There are many examples of Hom-Lie algebras that are generalizations of the one given in Theorem \ref{derhomliealg}:
  for all $m,l\in\mathbb{Z}$, $(\Der_{\sigma^m,\sigma^m}(\huaA),[\cdot,\cdot]_{ml},\Ad_{\sigma^l})$ is a Hom-Lie algebra,
where the skew-symmetric bilinear map $[\cdot,\cdot]_{ml}:\wedge^2\Der_{\sigma^m,\sigma^m}(\huaA)\longrightarrow\Der_{\sigma^m,\sigma^m}(\huaA)$ and $\Ad_{\sigma^l}$ are given by \begin{eqnarray}
[D_1,D_2]_{ml}&=&\sigma^l\circ D_1\circ\sigma^{-m}\circ D_2\circ\sigma^{-l}-\sigma^l\circ D_2\circ\sigma^{-m}\circ D_1\circ\sigma^{-l},\\
\Ad_{\sigma^l}(D)&=&\sigma^l\circ D\circ\sigma^{-l},\quad \forall D_1,D_2, D\in \Der_{\sigma,\sigma}(\huaA).
\end{eqnarray}
In particular, if $m=0,~l=0$, we recover the usual derivation Lie algebra.
\end{rmk}

\begin{ex}{\rm
  Let $M$ be a differential manifold, and $\varphi:M\longrightarrow M$ a diffeomorphism. Obviously, $\varphi^*:C^\infty(M)\longrightarrow C^\infty(M)$ is an algebra isomorphism. A section $X\in\Gamma(\varphi^*TM)$ of the pull back bundle $\varphi^*TM$ can be identified with a $(\varphi^*,\varphi^*)$-derivation on $C^\infty(M)$. Consequently, $(\Gamma(\varphi^*TM),[\cdot,\cdot]_{\varphi^*},\Ad_{\varphi^*})$ is a Hom-Lie algebra. This example plays important role in our future study on Hom-Lie algebroids.
  }
\end{ex}

\begin{ex}{\rm
Let  $\sigma:\mathbb{C}[[t]]\rightarrow\mathbb{C}[[t]]$ be an algebra isomorphism.
Then the set of $(\sigma,\sigma)$-derivations is given by
\begin{eqnarray*}
\Der_{\sigma,\sigma}(\mathbb{C}[[t]])=\{D|D(f(t))=f^\prime(s)|_{s=\sigma(t)}D_0(t), ~\text{for some}~D_0(t)\in\mathbb{C}[[t]]\}\cong\mathbb{C}[[t]],
\end{eqnarray*}
which has a basis $\{d_n|n\in\mathbb{Z}\}$, where
\begin{eqnarray*}
d_n(f(t))=f^\prime(s)|_{s=\sigma(t)}t^n.
\end{eqnarray*}
 In particular, we have
\begin{eqnarray}\label{17}
d_n(t^{k})=k\big{(}\sigma(t)\big{)}^{k-1}t^n.
\end{eqnarray}
The Hom-Lie bracket is given by
\begin{eqnarray}\label{18}
~[d_n,d_m]_\sigma(t^k)=k(m-n)\big{(}\sigma(t)\big{)}^{n+m+k-2}\Big{(}\big{(}\sigma^{-1}(t)\big{)}^\prime\Big{)}^2.
\end{eqnarray}

\begin{itemize}
\item[\rm(i)] If $\sigma(t)=qt$, for some $q\in\mathbb{C}\backslash\{0,1\}$, then
\begin{eqnarray*}
~[d_n,d_m]_\sigma&=&(m-n)q^{n+m-3}d_{n+m-1},\\
\Ad_{\sigma}(d_n)&=&q^{n-1}d_n.
\end{eqnarray*}
\item[\rm(ii)] If $\sigma(t)=qt^{-1}$, for some $q\in\mathbb{C}\backslash\{0\}$, then
\begin{eqnarray*}
~[d_n,d_m]_\sigma&=&(m-n)q^{n+m+1}d_{-n-m-3},\\
\Ad_{\sigma}(d_n)&=&-q^{n+1}d_{-n-2}.
\end{eqnarray*}
\end{itemize}
}
 \end{ex}

 \section{Dual  representations}

 In this section, we construct the dual representation of a representation of a Hom-Lie algebra without any additional condition. This is nontrivial. To our knowledge, people need to add a very strong condition to obtain a representation on the dual space in the former study. This restricts its development.

 \begin{defi} A {\bf representation} of a Hom-Lie algebra $(\g,[\cdot,\cdot]_\g,\phi)$ on
  a vector space $V$ with respect to $\beta\in\gl(V)$ is a linear map
  $\rho:\g\longrightarrow \gl(V)$, such that for all
  $x,y\in \g$, the following equalities are satisfied:
  \begin{eqnarray}
\rho(\phi(x))\circ \beta&=&\beta\circ \rho(x);\\
    \rho([x,y]_\g)\circ
    \beta&=&\rho(\phi(x))\circ\rho(y)-\rho(\phi(y))\circ\rho(x).
  \end{eqnarray}
  \end{defi}

  We denote a representation of a Hom-Lie algebra  $(\g,[\cdot,\cdot]_\g,\phi)$ by $(V,\beta,\rho)$.

  Let $(\g,[\cdot,\cdot]_\g,\phi)$ be a Hom-Lie algebra. The linear map $\phi:\g\longrightarrow\g$ can be extended to a linear map from $\wedge^k\g\longrightarrow\wedge^k\g$, for which we use the same notation $\phi$ via
$$
\phi(x_1\wedge\cdots\wedge x_k)=\phi(x_1)\wedge\cdots\wedge\phi(x_k).
$$  Furthermore, the bracket operation $[\cdot,\cdot]_\g$ can also be extended to $\wedge^\bullet\frkg$
via
\begin{equation}\label{eq:HG}
[x_1\wedge\cdots \wedge x_m,y_1\wedge \cdots
\wedge y_n]_\g=\sum_{i,j}(-1)^{i+j}[x_i,y_j]_\g \wedge
\phi(x_1\wedge\cdots \widehat{x_i}\cdots \wedge
x_m\wedge y_1\wedge\cdots
\widehat{y_j}\cdots\wedge y_n),
\end{equation}
for all $x_1\wedge\cdots \wedge x_m\in\wedge^m \g,~y_1\wedge \cdots
\wedge y_n\in\wedge^n \g$. Consequently, $(\oplus_k \wedge^k\g,\wedge,[\cdot,\cdot]_\g,\phi)$ is a Hom-Gerstenhaber algebra introduced in \cite{LGT}.

  \begin{ex}\label{ex:ad}{\rm
Let $(\g,[\cdot,\cdot]_\g,\phi)$ be a Hom-Lie algebra. For any integer $s$, the $\phi^s$-adjoint representation of $\g$ on $\wedge^k\g$, which we denote by $\ad^s$, is defined by
$$\ad^s_xY=[\phi^s(x),Y]_\g,\quad\forall x\in \g,~Y\in\wedge^k\g.$$
In particular, we write $\ad^0$ simply by $\ad$.
}
\end{ex}

Given a representation $\rho$ of the Hom-Lie algebra $(\g,[\cdot,\cdot]_\g,\phi)$ on $V$ with respect to $\beta$, define
 $\dM_\rho:\Hom(\wedge^k\g,V)\longrightarrow
\Hom(\wedge^{k+1}\g,V)$  by
\begin{eqnarray}
  &&\dM_\rho f(x_1,\cdots,x_{k+1})=\sum_{i=1}^{k+1}(-1)^{i+1}\rho(x_i)\big{(}f(\phi^{-1}(x_1),\cdots,\widehat{\phi^{-1}(x_i)},\cdots,\phi^{-1}(x_{k+1}))\big{)}\nonumber\\
  &&\qquad+\sum_{i<j}(-1)^{i+j}\beta f([\phi^{-2}(x_i),\phi^{-2}(x_j)]_\g,\phi^{-1}(x_1)\cdots,\widehat{\phi^{-1}(x_i)},\cdots,\widehat{\phi^{-1}(x_j)},\cdots,\phi^{-1}(x_{k+1})).\nonumber
\end{eqnarray}
Then we have $\dM^2_\rho=0.$ See \cite{caisheng} for more details. We use $H^\bullet(\g,\rho)$ to denote the corresponding cohomology group.

\begin{ex}\label{ex:derivation}{\rm
  Consider the $\phi^{-2}$-adjoint representation of $\g$ on $\wedge^2\g$. A linear map $\theta:\g\longrightarrow\wedge^2\g$ satisfying $\theta\circ \phi=\phi\circ\theta$ is a $1$-cocycle, i.e. $\dM_{\ad^{-2}}\theta=0$, if
  $$
  [\phi^{-2}(x),\theta\big{(}\phi^{-1}(y)\big{)}]_\g- [\phi^{-2}(y),\theta\big{(}\phi^{-1}(x)\big{)}]_\g- \theta([\phi^{-1}(x),\phi^{-1}(y)]_\g)=0.
  $$
  Equivalently, $\theta$ is a $(\phi^{-1},\phi^{-1})$-derivation.
  }
\end{ex}

Let $(V,\beta,\rho)$ be a representation of a Hom-Lie algebra $(\g,[\cdot,\cdot]_\g,\phi)$. In the sequel, we always assume that $\beta$ is invertible. Define $\rho^*:\g\longrightarrow\gl(V^*)$ as usual by
$$\langle \rho^*(x)(\xi),u\rangle=-\langle\xi,\rho(x)(u)\rangle,\quad\forall x\in\g,u\in V,\xi\in V^*.$$
However, in general $\rho^*$ is not a representation of $\g$ anymore (see \cite{BM} for details). Define $\rho^\star:\g\longrightarrow\gl(V^*)$ by
\begin{equation}\label{eq:new1}
 \rho^\star(x)(\xi):=\rho^*(\phi(x))\big{(}(\beta^{-2})^*(\xi)\big{)},\quad\forall x\in\g,\xi\in V^*.
\end{equation}
More precisely, we have
\begin{eqnarray}\label{eq:new1gen}
\langle\rho^\star(x)(\xi),u\rangle=-\langle\xi,\rho(\phi^{-1}(x))(\beta^{-2}(u))\rangle,\quad\forall x\in\g, u\in V, \xi\in V^*.
\end{eqnarray}
\begin{lem}\label{lem:dualrep}
 Let $(V,\beta,\rho)$ be a representation of a Hom-Lie algebra $(\g,[\cdot,\cdot]_\g,\phi)$. Then $\rho^\star:\g\longrightarrow\gl(V^*)$ defined above by \eqref{eq:new1} is a representation of $(\g,[\cdot,\cdot]_\g,\phi)$ on $V^*$ with respect to $(\beta^{-1})^*$.
\end{lem}
\begin{proof}
For all $x\in\g,\xi\in V^*$,  we have
\begin{eqnarray*}
\rho^\star(\phi(x))((\beta^{-1})^*(\xi))=\rho^*(\phi^{2}(x))(\beta^{-3})^*(\xi)=(\beta^{-1})^*(\rho^*(\phi(x))(\beta^{-2})^*(\xi))=(\beta^{-1})^*(\rho^\star(x)(\xi)),
\end{eqnarray*}
which implies $\rho^\star\big{(}\phi(x)\big{)}\circ(\beta^{-1})^*=(\beta^{-1})^*\circ\rho^\star(x)$.

On the other hand, by the Hom-Jacobi identity, for all $x,y\in\g,\xi\in V^*$ and $u\in V$, we have
\begin{eqnarray*}
&&\langle\rho^\star([x,y]_\g)((\beta^{-1})^*(\xi)),u\rangle\\
&=&\langle\rho^*(\phi ([x,y]_\g))((\beta^{-3})^*(\xi)),u\rangle\\
&=&-\langle(\beta^{-3})^*(\xi),\rho(\phi ([x,y]_\g))(u)\rangle\\
&=&-\langle(\beta^{-3})^*(\xi),\rho(\phi^{2}(x))(\rho(\phi (y))(\beta^{-1}(u))) -\rho(\phi^{2}(y))(\rho(\phi(x))(\beta^{-1}(u)))\rangle\\
&=&-\langle(\beta^{-4})^*(\xi),\rho(\phi^{3}(x))(\rho(\phi^{2}(y))(u))-\rho(\phi^{3}(y))(\rho(\phi^{2}(x))(u))\rangle\\
&=&-\langle\rho^*(\phi^{2}(y))(\rho^*(\phi^{3}(x))((\beta^{-4})^*(\xi)))-\rho^*(\phi^{2}(x))(\rho^*(\phi^{3}(y))((\beta^{-4})^*(\xi))),u\rangle\\
&=&-\langle\rho^\star(\phi (y))\big{(}\rho^\star(x)(\xi)\big{)}-\rho^\star(\phi (x))\big{(}\rho^\star(y)(\xi)\big{)},u\rangle,
\end{eqnarray*}
which implies that $$\rho^\star([x,y]_\g)\circ(\beta^{-1})^*=\rho^\star\big{(}\phi(x)\big{)}\circ\rho^\star(y)-\rho^\star\big{(}\phi(y)\big{)}\circ\rho^\star(x).$$
Therefore, $\rho^\star$ is a representation of  $(\g,[\cdot,\cdot]_\g,\phi)$ on $V^*$ with respect to $(\beta^{-1})^*$.
\end{proof}

\begin{rmk}\label{rmk:dualrep}
  The representation $\rho^\star$ given by \eqref{eq:new1} is the right dual representation in the Hom-Lie algebra context. It can be obtained essentially as follows. Let $(\g,\{\cdot,\cdot\})$ be a Lie algebra and $\overline{\rho}:\g\longrightarrow\gl(V)$  a representation of the Lie algebra $(\g,\{\cdot,\cdot\})$ on $V$, and $\overline{\rho}^*:\g\longrightarrow\gl(V^*)$ the associated dual representation. Then we have two semidirect product Lie algebras $(\g\oplus V,\{\cdot,\cdot\}_{\overline{\rho}})$ and $(\g\oplus V^*,\{\cdot,\cdot\}_{\overline{\rho}^*})$. Note that $(\g\oplus V,\{\cdot,\cdot\}_{\overline{\rho}})$ can be viewed as an abelian extension of $\g$ by $V$. Let $(\phi,\beta)$ be an inducible pair (\cite{BS}), i.e. $\phi:\g\longrightarrow \g$ is a Lie algebra automorphism and $\beta\in GL(V)$ satisfy
   $$
   \beta\overline{\rho}(x)=\overline{\rho}(\phi(x))\beta,\quad \forall x\in\g.
  $$
  Then $\phi+\beta$ is an automorphism of the Lie algebra $(\g\oplus V,\{\cdot,\cdot\}_{\overline{\rho}})$. Since $\phi$ is a Lie algebra automorphism, it follows that $(\g,[\cdot,\cdot]_\phi,\phi)$ is a Hom-Lie algebra (\cite{Yao0}), where $[\cdot,\cdot]_\phi:\wedge^2\g\longrightarrow\g$ is given by $[x,y]_\phi=\phi\{x,y\}$. Similarly, we obtain a Hom-Lie algebra $(\g\oplus V,[\cdot,\cdot]_{\phi+\beta},\phi+\beta)$. Now we define $\rho:\g\longrightarrow\gl(V)$ by
  $$\rho(x)(u)=[x,u]_{\phi+\beta}=\beta\overline{\rho}(x)(u)=\overline{\rho}(\phi(x))\beta(u).$$
  It follows that $\rho$ is a representation of the Hom-Lie algebra $(\g,[\cdot,\cdot]_\phi,\phi)$ on $V$ with respect to $\beta$.

  Now using this principle, we investigate the dual representation. First we observe that we should consider $ (\beta^{-1})^*:\g^*\longrightarrow\g^*$ rather than $\beta^*$ to ensure that the following equality holds:
  $$
  (\beta^{-1})^*\overline{\rho}^*(x)=\overline{\rho}^*(\phi(x)) (\beta^{-1})^*.
  $$
  Then $\phi+(\beta^{-1})^*$ is an automorphism of the Lie algebra $(\g\oplus V^*,\{\cdot,\cdot\}_{\overline{\rho}^*})$, and we obtain a Hom-Lie algebra $(\g\oplus V^*,[\cdot,\cdot]_{\phi+(\beta^{-1})^*},\phi+(\beta^{-1})^*)$.  Now we define $\rho^\star:\g\longrightarrow\gl(V^*)$ by
  $$\rho^\star(x)(\xi)=[x,\xi]_{\phi+(\beta^{-1})^*}=(\beta^{-1})^*\overline{\rho}^*(x)(\xi)=\overline{\rho}^*(\phi(x))(\beta^{-1})^*(\xi).$$
  It follows that $\rho^\star$ is a representation of the Hom-Lie algebra $(\g,[\cdot,\cdot]_\phi,\phi)$ on $V^*$ with respect to $(\beta^{-1})^*$.

  Finally, one can deduce that the relation between $\rho$ and $\rho^\star$ is exactly given by \eqref{eq:new1gen}. This justifies that $\rho^\star$ is the right dual representation in the Hom-Lie algebra context.
\end{rmk}

  \begin{rmk}\label{rhostar}
  In Example \ref{ex:ad}, we have seen that there are a series of adjoint representation. Parallel to this result, for any $s\in \mathbb Z$, we can
  define $ \rho_s^\star:\g\longrightarrow\gl(V^*)$ by  $$ \rho_s^\star(x)(\xi)=\rho^*(\phi^s(x))\big{(}(\beta^{-2})^*(\xi)\big{)}.$$  Then $\rho_s^\star$ is also a representation of $(\g,[\cdot,\cdot]_\g,\phi)$ on $V^*$ with respect to $(\beta^{-1})^*$.
  \end{rmk}

\begin{lem}\label{lem:dualdual}
Let $(V,\beta,\rho)$ be a representation of Hom-Lie algebra  $(\g,[\cdot,\cdot]_\g,\phi)$. Then we have $$(\rho^\star)^\star=\rho.$$
\end{lem}
\begin{proof}
For all $x\in\g,u\in V,\xi\in V^*$, we have
\begin{eqnarray*}
\langle (\rho^\star)^\star(x)(u),\xi\rangle
&=&\langle(\rho^\star)^*(\phi(x))(\beta^2(u)),\xi\rangle
=-\langle\beta^2(u),\rho^\star(\phi(x))(\xi)\rangle\\
&=&-\langle\beta^2(u),\rho^*(\phi^2(x))((\beta^{-2})^*(\xi))\rangle
=\langle\rho(\phi^2(x))(\beta^2(u)),(\beta^{-2})^*(\xi)\rangle\\
&=&\langle\rho(x)(u),\xi\rangle,
\end{eqnarray*}
which implies that $(\rho^\star)^\star=\rho$.
\end{proof}

\begin{cor}\label{lem:rep}
 Let $(\g,[\cdot,\cdot]_\g,\phi)$ be a Hom-Lie algebra. Then $\ad^\star:\g\longrightarrow\gl(\g^*)$  defined by
 \begin{equation}
  \ad^\star_x\xi=\ad^*_{\phi(x)}(\phi^{-2})^*(\xi),\quad \forall x\in\g, \xi\in\g^*,
 \end{equation}
 is a representation of the Hom-Lie algebra $(\g,[\cdot,\cdot]_\g,\phi)$ on $\g^*$ with respect to $(\phi^{-1})^*$, which is called the {\bf coadjoint representation}.
\end{cor}

Using the coadjoint representation $\ad^\star$, we can obtain a semidirect product Hom-Lie algebra structure on $\g\oplus\g^*$.
\begin{cor}
Let $(\g,[\cdot,\cdot]_\g,\phi)$ be a Hom-Lie algebra. Then there is a natural Hom-Lie algebra $(\g\oplus\g^*,[\cdot,\cdot]_s,\phi\oplus(\phi^{-1})^*)$, where the Hom-Lie bracket $[\cdot,\cdot]_s$ is given by
\begin{equation}\label{eq:semidirectproduct}
  [x+\xi,y+\eta]_s=[x,y]_\g+\ad_{x}^\star\eta-\ad_{y}^\star\xi=[x,y]_\g+\ad_{\phi (x)}^*(\phi^{-2})^*(\eta)-\ad_{\phi (y)}^*(\phi^{-2})^*(\xi),
\end{equation}
for all $x,y\in\g, \xi,\eta\in\g^*$.
\end{cor}

  The above construction of the semidirect product Hom-Lie algebra $(\g\oplus\g^*,[\cdot,\cdot]_s,\phi\oplus(\phi^{-1})^*)$ is a very important step to build a Manin triple from a purely Hom-Lie bialgebra as shown in the next section.

 \section{Purely Hom-Lie bialgebras and Manin triples for Hom-Lie algebras}

 In this section, we give the notions of a Manin triple for Hom-Lie algebras and a purely Hom-Lie bialgebra, and show that they are equivalent. First we modify the notion of a quadratic Hom-Lie algebra.

 \begin{defi}
   Let $(\huaV,[\cdot,\cdot]_\huaV,\phi_\huaV)$ be a Hom-Lie algebra. A symmetric bilinear form $B$ on $\huaV$ is called invariant if for all $u,v,w\in\huaV$, we have
   \begin{eqnarray}
   \label{eq:inv1} B([u,v]_\huaV,\phi_\huaV(w)) &=&-B(\phi_\huaV(v),[u,w]_\huaV), \\
    \label{eq:inv2}    B(\phi_\huaV(u),\phi_\huaV(v)) &=&B( u,v).
   \end{eqnarray}
   The quadruple $(\huaV,[\cdot,\cdot]_\huaV,\phi_\huaV, B)$  is called a {\bf quadratic Hom-Lie algebra}.
 \end{defi}

\begin{ex}\label{ex:quadratic}{\rm
  Let $\{e_1,e_2,e_3,e_4\}$ be a basis of a 4-dimensional vector space $\huaV$. Define a skew-symmetric bracket operation $[\cdot,\cdot]_\huaV$ on $\huaV$ by
 \begin{eqnarray*}
  &&[e_1,e_2]_\huaV=e_2, \quad [e_1,e_3]_\huaV=e_2,\quad [e_1,e_4]_\huaV=-e_1-e_2+e_3-e_4,\\
  &&[e_2,e_3]_\huaV=0,\quad[e_2,e_4]_\huaV=e_3,\quad [e_3,e_4]_\huaV=e_3.
\end{eqnarray*}
  Then $(\g,[\cdot,\cdot]_\g,\phi_\huaV)$ is a Hom-Lie algebra in which $\phi_\huaV=\left(\begin{array}{cccc}1&1&0&0\\0&1&0&0\\
  0&0&1&0\\0&0&-1&1\end{array}\right)$. More precisely, $$\phi_\huaV(e_1)=e_1+e_2,\quad\phi_\huaV(e_2)=e_2,\quad\phi_\huaV(e_3)=e_3,\quad\phi_\huaV(e_4)=-e_3+e_4.$$ Furthermore, let $(b_{ij})=\left(\begin{array}{cccc}0&0&1&0\\0&0&0&1\\
  1&0&0&0\\0&1&0&0\end{array}\right)$ and define the symmetric bilinear form via $B(e_i,e_j)=b_{ij}$. Then $(\huaV,[\cdot,\cdot]_\huaV,\phi_\huaV,B)$ is a quadratic Hom-Lie algebra.
  }
\end{ex}

 \begin{rmk}\label{rmk:B}
   Note that the invariant condition \eqref{eq:inv2} is not the same as the one give in \cite{BM}, where the author used the condition
   \begin{equation}\label{eq:invinf}
   B(\phi_\huaV(u),v) =B( u,\phi_\huaV(v)).
   \end{equation}
   Note that the compatibility condition between $\phi_\g$ and the multiplication $[\cdot,\cdot]_\g$ is that $\phi_\g$ is an algebra automorphism. For the algebra $(\g,[\cdot,\cdot]_\g)$, the set of algebra automorphisms is a group, called the automorphism group. Thus, it is natural to require that $\phi_\g$ and the bilinear form $B$ satisfy a similar compatibility condition. One can show that the set of linear maps that preserve the bilinear form $B$ in the sense of \eqref{eq:inv2} is a group. Actually it is the orthogonal group associated to the bilinear form $B$.  However, linear maps that preserve the bilinear form $B$ in the sense of \eqref{eq:invinf} do not have such properties.
 \end{rmk}

 \begin{defi}
  A \textbf{Manin triple for Hom-Lie algebras} is a triple $(\huaV;\g,\g')$ in which $(\huaV, [\cdot,\cdot]_\huaV,\\ \phi_\huaV,  B)$ is a quadratic Hom-Lie algebra,    $(\g,[\cdot,\cdot]_{\g},\phi_{\g})$ and $(\g^\prime,[\cdot,\cdot]_{\g^\prime},\phi_{\g^\prime})$ are   isotropic Hom-Lie sub-algebras of $\huaV$ such that \begin{itemize}
    \item[\rm(i)] $\huaV=\g\oplus \g^\prime$ as vector spaces;
      \item[\rm(ii)] $\phi_\huaV=\phi_\g\oplus\phi_{\g^\prime}$.
  \end{itemize}
 \end{defi}

 \begin{ex}\label{ex:Manintriple}{\rm
Let  $(\huaV,[\cdot,\cdot]_\huaV,\phi_\huaV,B)$ be the quadratic Hom-Lie algebra given in Example \ref{ex:quadratic}. Let $(\g,[\cdot,\cdot]_{\g},\phi_{\g})$ be the Hom-Lie subalgebra given by
$$
\g=span\{e_1,e_2\},\quad [e_1,e_2]_\g=e_2,\quad \phi_\g=\left(\begin{array}{cc}1&1\\0&1\end{array}\right).
$$Let $(\g',[\cdot,\cdot]_{\g'},\phi_{\g'})$ be the Hom-Lie subalgebra given by
$$
\g'=span\{e_3,e_4\},\quad [e_3,e_4]_{\g'}=e_3,\quad \phi_{\g'}=\left(\begin{array}{cc}1&0\\-1&1\end{array}\right).
$$
Then it is straightforward to see that $(\huaV;\g,\g')$ is a Manin triple for Hom-Lie algebras.
}
 \end{ex}

 \begin{defi}
 Let $(\g,[\cdot,\cdot]_{\g},\phi)$ and $(\g^*,[\cdot,\cdot]_{\g^*},(\phi^{-1})^*)$ be two Hom-Lie algebras.  $(\g,\g^*)$ is called a {\bf purely Hom-Lie bialgebra} if the following compatibility condition holds:
 \begin{eqnarray}\label{eq:homLiebi}
 \Delta([x,y]_{\g})=\ad_{\phi^{-1}(x)}\Delta (y)-\ad_{\phi^{-1}(y)}\Delta (x),
 \end{eqnarray}
 where $\Delta:\g\longrightarrow \wedge^2\g$ is the dual of the Hom-Lie algebra structure $[\cdot,\cdot]_{\g^*}:\wedge^2\g^*\longrightarrow \g^*$ on $\g^*$, i.e.
 $$
 \langle\Delta(x),\xi\wedge\eta\rangle=\langle x,[\xi,\eta]_{\g^*}\rangle.
 $$
 \end{defi}

 \begin{ex}{\rm
 Let $\g$ and $\g'$ be the two Hom-Lie algebras given in Example \ref{ex:Manintriple}. We view $\g'$ as the dual space of $\g$ and view $\{e_3,e_4\}$ the dual basis of $\{e_1,e_2\}$. Then $(\g,\g')$ is a purely Hom-Lie bialgebra.
 }
 \end{ex}
\begin{rmk}
   We give some explanations of the compatibility condition \eqref{eq:homLiebi}. Due to the fact that $\oplus_{k=0}^\infty \wedge^k\g$ is a graded Hom-Lie algebra, \eqref{eq:homLiebi} means that $\Delta$ is a $(\phi^{-1},\phi^{-1})$-derivation. On the other hand, using the cohomology point of view,  \eqref{eq:homLiebi} also means that $\Delta $ is a $1$-cocycle on $\g$ with the coefficients in $(\wedge^2\g,\ad^{-2})$.
\end{rmk}

\begin{rmk}
Our definition of a purely Hom-Lie bialgebra is different from the Hom-Lie bialgebras given in \cite{Yao3} and \cite{sheng1} respectively. In \cite{Yao3}, a Hom-Lie bialgebra is  a quadruple $(\g,[\cdot,\cdot]_\g,\Delta,\phi)$, where $(\g,[\cdot,\cdot]_\g,\phi)$ is a Hom-Lie algebra and $(\g, \Delta,\phi)$ is a Hom-Lie coalgebra\footnote{$(\g, \Delta,\phi)$ is a Hom-Lie coalgebra if and only if $(\g^*,\Delta^*,\phi^*)$ is a Hom-Lie algebra.} such that the following compatibility condition is satisfied:
\begin{equation}\label{eq:oldcom}
\Delta([x,y]_\g)=\ad_{\phi(x)}\Delta(y)-\ad_{\phi(y)}\Delta(x),\quad\forall x,y\in\g.
\end{equation}
Note that even though one can also explain the above compatibility condition as a derivation, or a $1$-cocycle, there is not a  natural Hom-Lie algebra structure on $\g\oplus \g^*$. To solve this problem, the authors introduced another kind of Hom-Lie bialgebras based on admissible Hom-Lie algebras in \cite{sheng1}. However, the  condition for an admissible Hom-Lie algebra is quite strong, which makes such kind of Hom-Lie bialgebras too restrictive. It is obvious that the compatibility conditions \eqref{eq:homLiebi} and \eqref{eq:oldcom} are not the same. The other difference, which is more intrinsic, is that the algebra homomorphism on $\g^*$ is $(\phi^{-1})^*$ in a purely Hom-Lie bialgebra and the algebra homomorphism on $\g^*$ is $\phi^*$ in Hom-Lie bialgebras given in  \cite{Yao3} and \cite{sheng1}. By Remark \ref{rmk:dualrep}, we believe that using $(\phi^{-1})^*$ as the algebra homomorphism on $\g^*$ is more natural.
\end{rmk}
\begin{lem}\label{mudelta}
Let $(\g,[\cdot,\cdot]_{\g},\phi)$ and $(\g^*,[\cdot,\cdot]_{\g^*},(\phi^{-1})^*)$ be two Hom-Lie algebras.  Then the condition \eqref{eq:homLiebi} is equivalent to
\begin{equation}\label{ups}
  \Upsilon([\xi,\eta]_{\g^*})=\AD_{\phi^*(\xi)}\Upsilon(\eta)-\AD_{\phi^*(\eta)}\Upsilon(\xi),\quad\forall \xi,\eta\in\g^*,
\end{equation}
 where $\AD_{\xi}\eta=[\xi,\eta]_{\g^*}$ and $\Upsilon:\g^*\longrightarrow \wedge^2\g^*$ is the dual of the Hom-Lie algebra structure $[\cdot,\cdot]_\g:\wedge^2\g\longrightarrow \g$ on $\g$, i.e.
 $$
 \langle\Upsilon(\xi),x\wedge y\rangle=\langle \xi,[x,y]_\g\rangle.
 $$
\end{lem}
\begin{proof} First for all $x\in \g,\xi_1,\cdots,\xi_k\in \g^*, \Xi=\xi_1\wedge\cdots\wedge\xi_k,$ we have
 \begin{eqnarray}
\label{adphi} (\phi^{-1})^*\ad_x^*\Xi&=&\ad_{\phi(x)}^*(\phi^{-1})^*(\Xi),\\
\label{5} \ad_x^*(\xi_1\wedge\cdots\wedge\xi_k)&=&\sum_{i=1}^k\phi^*(\xi_1)\wedge\cdots\wedge\ad_x^*\xi_i\wedge\cdots\wedge\phi^*(\xi_k).
 \end{eqnarray}
Then for all $x,y\in\g,~\xi,\eta\in\g^*$, by $\langle\AD_\xi^*x,\eta\rangle=-\langle x,\AD_\xi\eta\rangle$, we have
\begin{eqnarray*}
&&\langle-\Delta([x,y]_\g)+\ad_{\phi^{-1}(x)}\Delta (y)-\ad_{\phi^{-1}(y)}\Delta (x),\xi\wedge\eta\rangle\\
&=&-\langle[x,y]_\g,[\xi,\eta]_{\g^*}\rangle-\langle\Delta (y),\ad_{\phi^{-1}(x)}^*(\xi\wedge\eta)\rangle+\langle\Delta (x),\ad_{\phi^{-1}(y)}^*(\xi\wedge\eta)\rangle\\
&=&-\langle x\wedge y,\Upsilon([\xi,\eta]_{\g^*})\rangle
-\langle\Delta (y),\ad_{\phi^{-1}(x)}^*\xi\wedge\phi^*(\eta)+\phi^*(\xi)\wedge\ad_{\phi^{-1}(x)}^*\eta\rangle\\
&&+\langle\Delta (x),\ad_{\phi^{-1}(y)}^*\xi\wedge\phi^*(\eta)+\phi^*(\xi)\wedge\ad_{\phi^{-1}(y)}^*\eta\rangle\\
&=&-\langle x\wedge y,\Upsilon([\xi,\eta]_{\g^*})\rangle
+\langle y,\AD_{\phi^*(\eta)}\ad_{\phi^{-1}(x)}^*\xi\rangle
-\langle y,\AD_{\phi^*(\xi)}\ad_{\phi^{-1}(x)}^*\eta\rangle\\
&&-\langle x,\AD_{\phi^*(\eta)}\ad_{\phi^{-1}(y)}^*\xi\rangle
+\langle x,\AD_{\phi^*(\xi)}\ad_{\phi^{-1}(y)}^*\eta\rangle\\
&=&-\langle x\wedge y,\Upsilon([\xi,\eta]_{\g^*})\rangle
+\langle\ad_{\phi^{-1}(x)}\AD_{\phi^*(\eta)}^*y,\xi\rangle
-\langle\ad_{\phi^{-1}(x)}\AD_{\phi^*(\xi)}^*y,\eta\rangle\\
&&-\langle\ad_{\phi^{-1}(y)}\AD_{\phi^*(\eta)}^*x,\xi\rangle
+\langle\ad_{\phi^{-1}(y)}\AD_{\phi^*(\xi)}^*x,\eta\rangle\\
&=&-\langle x\wedge y,\Upsilon([\xi,\eta]_{\g^*})\rangle
+\langle\AD_{\phi^*(\eta)}^*(x\wedge y),\Upsilon(\xi)\rangle
-\langle\AD_{\phi^*(\xi)}^*(x\wedge y),\Upsilon(\eta)\rangle\\
&=&\langle x\wedge y,-\Upsilon([\xi,\eta]_{\g^*})-\AD_{\phi^*(\eta)}\Upsilon(\xi)+\AD_{\phi^*(\xi)}\Upsilon(\eta)\rangle,
\end{eqnarray*}
which implies that \eqref{eq:homLiebi} and \eqref{ups} are equivalent.
\end{proof}

For two Hom-Lie algebras $(\g,[\cdot,\cdot]_\g,\phi)$ and $(\g^*,[\cdot,\cdot]_{\g^*},(\phi^{-1})^*)$, define a skew-symmetric bilinear operation $[\cdot,\cdot]_{\bowtie}:\wedge^2(\g\oplus \g^*)\longrightarrow \g\oplus \g^*$ and a symmetric bilinear form $B$ by
\begin{eqnarray}
\label{eq:bracketdouble}[x+\xi,y+\eta]_{\bowtie}
&=&\big{(}[x,y]_\g+\AD_{\xi}^\star y-\AD_{\eta}^\star x\big{)}
+\big{(}[\xi,\eta]_{\g^*}+\ad_{x}^\star\eta-\ad_{y}^\star\xi\big{)}\nonumber\\
&=&\big{(}[x,y]_\g+\AD_{(\phi^{-1})^*(\xi)}^*\phi^2(y)-\AD_{(\phi^{-1})^*(\eta)}^*\phi^2(x)\big{)}\nonumber\\
&&+\big{(}[\xi,\eta]_{\g^*}+\ad_{\phi (x)}^*(\phi^{-2})^*(\eta)-\ad_{\phi (y)}^*(\phi^{-2})^*(\xi)\big{)}
\end{eqnarray}
and
\begin{equation}\label{eq:bilinear form}
  B(x+\xi,y+\eta)=\xi(y)+\eta(x).
\end{equation}
Note that \eqref{eq:bracketdouble} can be viewed as a generalization of \eqref{eq:semidirectproduct}.

 \begin{thm}\label{thm:equivalent}
Let $(\g,\g^*)$ be a purely Hom-Lie bialgebra. Then $\Big{(}\g\oplus \g^*,[\cdot,\cdot]_{\bowtie},\phi\oplus(\phi^{-1})^*,B\Big{)}$ is a quadratic Hom-Lie algebra, where the Hom-Lie bracket $[\cdot,\cdot]_{\bowtie}$ and the bilinear form $B$ are given by \eqref{eq:bracketdouble} and \eqref{eq:bilinear form} respectively. Furthermore, $(\g\oplus \g^*;\g,\g^*)$ is a Manin triple for Hom-Lie algebras.

Conversely, let $(\g,[\cdot,\cdot]_\g,\phi_\g)$ and $(\g^*,[\cdot,\cdot]_{\g^*},\phi_{\g^*})$  be two Hom-Lie algebras and $(\g\oplus \g^*, [\cdot,\cdot]_{\g\oplus \g^*}, \phi_{\g} \\
\oplus \phi_{\g^*},B)$  a quadratic Hom-Lie algebra, where the  symmetric bilinear form $B$ is given by \eqref{eq:bilinear form}.   If $(\g\oplus \g^*;\g,\g^*)$ is a Manin triple for Hom-Lie algebras, then $(\g,\g^*)$ is a purely Hom-Lie bialgebra.
\end{thm}
 \begin{proof}

 Let $(\g,\g^*)$ be a purely Hom-Lie bialgebra. For all $x,y\in \g$ and $\xi,\eta\in \g^*$, by  \eqref{eq:bracketdouble}, we have
 \begin{eqnarray}\label{eq:homomorphism}
 \big{(}\phi\oplus(\phi^{-1})^*\big{)}[x+\xi,y+\eta]_{\bowtie}=[ \big{(}\phi\oplus(\phi^{-1})^*\big{)}(x+\xi), \big{(}\phi\oplus(\phi^{-1})^*\big{)}(y+\eta)]_{\bowtie}.
 \end{eqnarray}

Now we show that $[\cdot,\cdot]_{\bowtie}$ satisfies the Hom-Jacobi identity:
\begin{eqnarray}\label{homjid}
 [[x+\xi,y+\eta]_{\bowtie},\big{(}\phi\oplus(\phi^{-1})^*\big{)}(z+\delta)]_{\bowtie}+c.p.=0,
 \end{eqnarray}
 which is equivalent to the following two equalities
    \begin{eqnarray}
    \label{23}[[x,y]_\g,(\phi^{-1})^*(\delta)]_{\bowtie}+[[y,\delta]_{\bowtie},\phi (x)]_{\bowtie}+[[\delta,x]_{\bowtie},\phi (y)]_{\bowtie}&=&0,\\
    ~\label{eq:dual}[[\xi,\eta]_{\g^*},\phi( z)]_{\bowtie}+[[\eta,z]_{\bowtie},(\phi^{-1})^*(\xi)]_{\bowtie}+[[z,\xi]_{\bowtie},(\phi^{-1})^*(\eta)]_{\bowtie}&=&0.
    \end{eqnarray}
By straightforward computation, we have
    \begin{eqnarray}
      &&[[x,y]_\g,(\phi^{-1})^*(\delta)]_{\bowtie}+[[y,\delta]_{\bowtie},\phi (x)]_{\bowtie}+[[\delta,x]_{\bowtie},\phi (y)]_{\bowtie}\nonumber\\
      &=&\big{(}-\AD_{(\phi^{-2})^*(\delta)}^*\phi^2([x,y]_\g)+\AD_{\ad_{\phi^2(y)}^*(\phi^{-3})^*(\delta)}^*\phi^3(x)-\AD_{\ad_{\phi^2(x)}^*(\phi^{-3})^*(\delta)}^*\phi^3(y)\nonumber\\
      &&-[\AD_{(\phi^{-1})^*(\delta)}^*\phi^2(y),\phi (x)]_\g+[\AD_{(\phi^{-1})^*(\delta)}^*\phi^2(x),\phi (y)]_\g\big{)}\nonumber\\
      &&+\big{(}\ad_{\phi([x,y]_\g)}^*(\phi^{-3})^*(\delta)-\ad_{\phi^2(x)}^*\circ\ad_{\phi^3(y)}^*(\phi^{-4})^*(\delta)+\ad_{\phi^2(y)}^*\circ\ad_{\phi^3(x)}^*(\phi^{-4})^*(\delta)\big{)}.\nonumber
          \end{eqnarray}
By Corollary \ref{lem:rep}, $\ad_x^\star(\xi)=\ad_{\phi (x)}^*(\phi^{-2})^*(\xi)$ is a representation of $\g$ on $\g^*$ with respect to $(\phi^{-1})^*$. Thus, we have
\begin{eqnarray*}
&&\ad_{\phi([x,y]_\g)}^*(\phi^{-3})^*(\delta)-\ad_{\phi^2(x)}^*\circ\ad_{\phi^3(y)}^*(\phi^{-4})^*(\delta)+\ad_{\phi^2(y)}^*\circ\ad_{\phi^3(x)}^*(\phi^{-4})^*(\delta)\\
&=&\ad_{[x,y]_\g}^\star\big{(}(\phi^{-1})^*\delta\big{)}-\ad_{\phi (x)}^\star\ad_{y}^\star\delta+\ad_{\phi (y)}^\star\ad_{x}^\star\delta\\
&=&0.
\end{eqnarray*}
 On the other hand, for any $\gamma\in \g^*$, we have
  \begin{eqnarray}
  &&\langle\gamma,-\AD_{(\phi^{-2})^*(\delta)}^*\phi^2([x,y]_\g)+\AD_{\ad_{\phi^2(y)}^*(\phi^{-3})^*(\delta)}^*\phi^3(x)-\AD_{\ad_{\phi^2(x)}^*(\phi^{-3})^*(\delta)}^*\phi^3(y)\nonumber\\
      &&\qquad-[\AD_{(\phi^{-1})^*(\delta)}^*\phi^2(y),\phi (x)]_\g+[\AD_{(\phi^{-1})^*(\delta)}^*\phi^2(x),\phi (y)]_\g\rangle\nonumber\\
      &=&\langle[(\phi^{-2})^*(\delta),\gamma]_{\g^*},\phi^2([x,y]_\g)\rangle-\langle[\ad_{\phi^2(y)}^*(\phi^{-3})^*(\delta),\gamma]_{\g^*},\phi^3(x)\rangle+\langle[\ad_{\phi^2(x)}^*(\phi^{-3})^*(\delta),\gamma]_{\g^*},\phi^3(y)\rangle\nonumber\\
      &&+\langle\AD_{(\phi^{-1})^*(\delta)}\circ\ad_{\phi (x)}^*(\gamma),\phi^2(y)\rangle-\langle\AD_{(\phi^{-1})^*(\delta)}\circ\ad_{\phi (y)}^*(\gamma),\phi^2(x)\rangle\nonumber\\
       &=&\langle\delta\wedge(\phi^2)^*(\gamma),\Delta([x,y]_\g)\rangle-\langle\ad_{\phi^{-1}(y)}^*\big{(}\delta\wedge(\phi^2)^*(\gamma)\big{)},\Delta(x)\rangle+\langle\ad_{\phi^{-1}(x)}^*\big{(}\delta\wedge(\phi^2)^*(\gamma)\big{)},\Delta(y)\rangle\nonumber\\
      &=&\langle\Delta([x,y]_\g)+\ad_{\phi^{-1}(y)}\Delta(x)-\ad_{\phi^{-1}(x)}\Delta(y),\delta\wedge(\phi^2)^*(\gamma)\rangle\nonumber\\
      &=&0,\nonumber
  \end{eqnarray}
which implies that \eqref{23} holds. Similarly, by Lemma \ref{mudelta}, we deduce that \eqref{eq:dual} holds.  Therefore, $[\cdot,\cdot]_{\bowtie}$ satisfies the Hom-Jacobi identity \eqref{homjid}.

By \eqref{eq:homomorphism} and \eqref{homjid}, $(\g\oplus\g^*,[\cdot,\cdot]_{\bowtie},\phi\oplus(\phi^{-1})^*)$ is a Hom-Lie algebra.
Furthermore, it is obvious that the bilinear form $B$ defined by \eqref{eq:bilinear form} is invariant. Therefore, $(\g\oplus\g^*;\g,\g^*)$ is Manin triple for Hom-Lie algebras.

Conversely, if $(\g\oplus \g^*;\g,\g^*)$ is a Manin triple for Hom-Lie algebras, then by  \eqref{eq:inv2}, we obtain $\phi_{\g^*}=(\phi_\g^{-1})^*$. By \eqref{eq:inv1},  we can deduce that
\begin{eqnarray}
[x+\xi,y+\eta]_{\g\oplus \g^*}&=&[x+\xi,y+\eta]_{\bowtie}.\nonumber
\end{eqnarray}
Then similarly as the proof of \eqref{23}, we get the compatibility condition \eqref{eq:homLiebi}. Thus, $(\g,\g^*)$ is a purely Hom-Lie bialgebra.
\end{proof}

 \section{Triangular Hom-Lie bialgebras}

For any $r\in\wedge^2\g$, the induced skew-symmetric linear map $r^\sharp:\g^*\rightarrow\g$ is defined by
\begin{eqnarray}
\langle r^\sharp(\xi),\eta\rangle=\langle r,\xi\wedge\eta\rangle.
\end{eqnarray}

\begin{defi}
A purely Hom-Lie bialgebra $(\g,\g^*)$ is said to be {\bf  coboundary} if
\begin{eqnarray}\label{eq:coboundary}
\Delta(x)=(\dM_{\ad^{-2}}r)(x)=[\phi^{-2}(x),r]_\g,\quad \mbox{for some} \quad r\in\wedge^2\g.
\end{eqnarray}
\end{defi}

\begin{pro}\label{coobracket}
Let $(\g,[\cdot,\cdot]_\g,\phi)$ be a Hom-Lie algebra and $\Delta:\g\longrightarrow \wedge^2\g$ defined by \eqref{eq:coboundary} for some $r\in \wedge^2\g$ satisfying
\begin{equation}\label{eq:conr}
r^\sharp\circ(\phi^{-1})^*=\phi \circ r^\sharp.
\end{equation}
Then for all $\xi,\eta\in\g^*$, we have
\begin{eqnarray}\label{xieta}
~[\xi,\eta]_{\g^*}=\ad_{r^\sharp(\xi)}^\star\eta-\ad_{r^\sharp(\eta)}^\star\xi=\ad_{\phi(r^\sharp(\xi))}^*(\phi^{-2})^*(\eta)-\ad_{\phi(r^\sharp(\eta))}^*(\phi^{-2})^*(\xi),
\end{eqnarray}
where $[\cdot,\cdot]_{\g^*}$ is defined by $\langle[\xi,\eta]_{\g^*},x\rangle=\langle\Delta(x),\xi\wedge\eta\rangle$. Furthermore, we have
 \begin{eqnarray}\label{eq:formu}
~[r^\sharp\circ\phi^*(\xi),r^\sharp\circ\phi^*(\eta)]_\g-r^\sharp\circ\phi^*([\xi,\eta]_{\g^*})=\frac{1}{2}[r,r]_\g(\xi,\eta).
\end{eqnarray}
\end{pro}
\begin{proof}
To be simple, assume $r=r_1\wedge r_2$. By \eqref{eq:coboundary} and \eqref{adphi}, we have
\begin{eqnarray*}
\langle[\xi,\eta]_{\g^*},x\rangle
&=&\langle\xi\wedge\eta,\Delta(x)\rangle
=\langle\xi\wedge\eta,[\phi^{-2}(x),r]_\g\rangle\\
&=&\langle\xi\wedge\eta,[\phi^{-2}(x),r_1]_\g\wedge\phi(r_2)+\phi(r_1)\wedge[\phi^{-2}(x),r_2]_\g\rangle\\
&=&\langle[\phi^{-2}(x),r_1]_\g,\xi\rangle\langle\phi(r_2),\eta\rangle
-\langle[\phi^{-2}(x),r_1]_\g,\eta\rangle\langle\phi(r_2),\xi\rangle\\
&&+\langle\phi(r_1),\xi\rangle\langle[\phi^{-2}(x),r_2]_\g,\eta\rangle
-\langle\phi(r_1),\eta\rangle\langle[\phi^{-2}x,r_2]_\g,\xi\rangle\\
&=&\langle[\phi^{-2}(x),\langle r_2,\phi^*(\eta)\rangle r_1]_\g,\xi\rangle
-\langle[\phi^{-2}(x),\langle r_2,\phi^*(\xi)\rangle r_1]_\g,\eta\rangle\\
&&+\langle[\phi^{-2}(x),\langle r_1,\phi^*(\xi)\rangle r_2]_\g,\eta\rangle
-\langle[\phi^{-2}(x),\langle r_1,\phi^*(\eta)\rangle r_2]_\g,\xi\rangle\\
&=&-\langle[\phi^{-2}(x),r^\sharp(\phi^*(\eta))]_\g,\xi\rangle
+\langle[\phi^{-2}(x),r^\sharp(\phi^*(\xi))]_\g,\eta\rangle\\
&=&\langle\phi^{-2}(x),-\ad_{r^\sharp(\phi^*(\eta))}^*\xi+\ad_{r^\sharp(\phi^*(\xi))}^*\eta\rangle\\
&=&\langle x,-\ad_{\phi^2(r^\sharp(\phi^*(\eta)))}^*(\phi^{-2})^*(\xi)
+\ad_{\phi^2(r^\sharp(\phi^*(\xi)))}^*(\phi^{-2})^*(\eta)\rangle,\\
&=&\langle x,\ad_{\phi(r^\sharp(\xi))}^*(\phi^{-2})^*(\eta)-\ad_{\phi(r^\sharp(\eta))}^*(\phi^{-2})^*(\xi)\rangle\\
&=&\langle x,\ad_{r^\sharp(\xi)}^\star\eta-\ad_{r^\sharp(\eta)}^\star\xi\rangle,
\end{eqnarray*}
which implies that \eqref{xieta} holds.

On the other hand,   for all $\theta\in\g^*$, we have
\begin{eqnarray*}
&&\langle[\phi(r^\sharp(\xi)),\phi(r^\sharp(\eta))]_\g-\phi(r^\sharp([\xi,\eta]_{\g^*})),\theta\rangle\\
&=&\langle[\phi(r^\sharp(\xi)),\phi(r^\sharp(\eta))]_\g,\theta\rangle
+\langle[\xi,\eta]_{\g^*},r^\sharp(\phi^*(\theta))\rangle\\
&=&\langle[\phi(r^\sharp(\xi)),\phi(r^\sharp(\eta))]_\g,\theta\rangle
+\langle\ad_{\phi(r^\sharp(\xi))}^*(\phi^{-2})^*(\eta)-\ad_{\phi(r^\sharp(\eta))}^*(\phi^{-2})^*(\xi),r^\sharp(\phi^*(\theta))\rangle\\
&=&\langle[\phi(r^\sharp(\xi)),\phi(r^\sharp(\eta))]_\g,\theta\rangle
-\langle(\phi^{-2})^*(\eta),[\phi(r^\sharp(\xi),r^\sharp(\phi^*(\theta)))]_\g\rangle\\
&&+\langle(\phi^{-2})^*(\xi),[\phi(r^\sharp(\eta)),r^\sharp(\phi^*(\theta))]_\g\rangle\\
&=&\langle[r^\sharp((\phi^{-1})^*(\xi)),r^\sharp((\phi^{-1})^*(\eta))]_\g,\theta\rangle
-\langle(\phi^{-2})^*(\eta),[r^\sharp((\phi^{-1})^*(\xi)),r^\sharp(\phi^*(\theta)]_\g\rangle\\
&&+\langle(\phi^{-2})^*(\xi),[r^\sharp((\phi^{-1})^*(\eta)),r^\sharp(\phi^*(\theta))]_\g\rangle\\
&=&-\langle r_1,(\phi^{-1})^*(\xi)\rangle\langle[r_2,r_1]_\g,\theta\rangle\langle r_2,(\phi^{-1})^*(\eta)\rangle
-\langle r_2,(\phi^{-1})^*(\xi)\rangle\langle[r_1,r_2]_\g,\theta\rangle\langle r_1,(\phi^{-1})^*(\eta)\rangle\\
&&+\langle r_1,(\phi^{-1})^*(\xi)\rangle\langle[r_2,r_1]_\g,(\phi^{-2})^*(\eta)\rangle\langle r_2,\phi^*(\theta)\rangle
+\langle r_2,(\phi^{-1})^*(\xi)\rangle\langle[r_1,r_2]_\g,(\phi^{-2})^*(\eta)\rangle\langle r_1,\phi^*(\theta)\rangle\\
&&-\langle r_1,(\phi^{-1})^*(\eta)\rangle\langle[r_2,r_1]_\g,(\phi^{-2})^*(\xi)\rangle\langle r_2,\phi^*(\theta)\rangle
-\langle r_2,(\phi^{-1})^*(\eta)\rangle\langle[r_1,r_2]_\g,(\phi^{-2})^*(\xi)\rangle\langle r_1,\phi^*(\theta)\rangle\\
&=&\frac{1}{2}[r,r]_\g((\phi^{-2})^*(\xi),(\phi^{-2})^*(\eta),\theta)\\
&=&\langle\frac{1}{2}[r,r]_\g((\phi^{-2})^*(\xi),(\phi^{-2})^*(\eta)),\theta\rangle,
\end{eqnarray*}
which implies that
\eqref{eq:formu}
holds.
\end{proof}

\begin{thm}\label{thm:coboundary}
Let $(\g,[\cdot,\cdot]_\g,\phi)$ be a  Hom-Lie algebra.  Then $(\g^*,[\cdot,\cdot]_{\g^*},(\phi^{-1})^*)$ is a Hom-Lie algebra, where $[\cdot,\cdot]_{\g^*}$ is defined by \eqref{xieta} for some $r\in\wedge^2\g$ satisfying \eqref{eq:conr}, if and only if
\begin{eqnarray}
\label{phir}
\ad_{x}[r,r]_\g&=&0,
\end{eqnarray}
  Under these conditions, $(\g,\g^*)$ is a coboundary purely Hom-Lie bialgebra.
\end{thm}
\begin{proof}
By straightforward computation, $(\phi^{-1})^*$ is an algebra automorphism, i.e. $(\phi^{-1})^*([\xi,\eta]_{\g^*})=[(\phi^{-1})^*(\xi),(\phi^{-1})^*(\eta)]_{\g^*}$ if and only if \eqref{eq:conr} holds.

By Proposition \ref{coobracket}, \eqref{phir} and the fact that $\ad^\star$ is a representation, we have
\begin{eqnarray*}
&&[(\phi^{-1})^*(\xi),[\eta,\delta]_{\g^*}]_{\g^*}+c.p.\\
&=&\ad_{\phi(r^\sharp(\xi))}^\star\Big{(}\ad_{r^\sharp(\eta)}^\star\delta
-\ad_{r^\sharp(\delta)}^\star\eta\Big{)}
-\ad_{\phi(r^\sharp([\eta,\delta]_{\g^*}))}^*(\phi^{-1})^*(\xi)+c.p.\\
&=&\ad_{\phi(r^\sharp(\xi))}^\star\big{(}\ad_{r^\sharp(\eta)}^\star(\delta)\big{)}
-\ad_{\phi(r^\sharp(\eta))}^\star\big{(}\ad_{r^\sharp(\xi)}^\star(\delta)\big{)}\\
&&-\ad_{r^\sharp([\xi,\eta]_{\g^*})}^\star(\phi^{-1})^*(\delta)+c.p.\\
&=&\ad_{[r^\sharp(\xi),r^\sharp(\eta)]_\g
-r^\sharp([\xi,\eta]_{\g^*})}^\star(\phi^{-1})^*(\delta).
\end{eqnarray*}
For all $x\in\g$, we have
\begin{eqnarray*}
&&\langle[(\phi^{-1})^*(\xi),[\eta,\delta]_{\g^*}]_{\g^*}+c.p.,x\rangle\\
&=&-\langle(\phi^{-3})^*\delta,\ad_{[\phi(r^\sharp(\xi)),\phi(r^\sharp(\eta))]_\g-\phi(r^\sharp([\xi,\eta]_{\g^*}))}x\rangle+c.p.\\
&=&-\langle\ad_x^*(\phi^{-3})^*(\delta),[\phi(r^\sharp(\xi)),\phi(r^\sharp(\eta))]_\g-\phi(r^\sharp([\xi,\eta]_{\g^*}))\rangle+c.p..
\end{eqnarray*}
Thus, by Proposition \ref{coobracket} and \eqref{5}, we have
\begin{eqnarray*}
&&\langle[(\phi^{-1})^*(\xi),[\eta,\delta]_{\g^*}]_{\g^*}+c.p.,x\rangle\\
&=&-\langle\ad_x^*(\phi^{-3})^*(\delta),\frac{1}{2}[r,r]_\g((\phi^{-2})^*(\xi),(\phi^{-2})^*(\eta))\rangle+c.p.\\
&=&-\frac{1}{2}[r,r]_\g((\phi^{-2})^*(\xi),(\phi^{-2})^*(\eta),\ad_x^*(\phi^{-3})^*(\delta))+c.p.\\
&=&-\frac{1}{2}\langle[r,r]_\g,\ad_x^*((\phi^{-3})^*(\xi)\wedge(\phi^{-3})^*(\eta)\wedge(\phi^{-3})^*(\delta))\rangle\\
&=&\frac{1}{2}\langle\ad_x[r,r]_\g,(\phi^{-3})^*(\xi\wedge\eta\wedge\theta)\rangle,
\end{eqnarray*}
which implies that Hom-Jacobi identity is satisfied if and only if $\ad_x[r,r]_\g=0,$ for all $x\in\g.$
\end{proof}

\begin{ex}{\rm
  Let $\{e_1,e_2,e_3\}$ be a basis of a 3-dimensional vector space $\g$.   Define a skew-symmetric bracket operation $[\cdot,\cdot]_\g$ on $\g$ by
  $$
  [e_1,e_3]=e_2.
  $$
  Then $(\g,[\cdot,\cdot]_\g,\phi)$ is a Hom-Lie algebra, where $\phi=\left(\begin{array}{ccc}a&0&0\\0&1&0\\0&0&\frac{1}{a}\end{array}\right).$ Let $r=e_1\wedge e_3$ and $r^\sharp=\left(\begin{array}{ccc}0&0&1\\0&0&0\\-1&0&0\end{array}\right)$. It is straightforward to see that \eqref{eq:conr} holds. Furthermore, we have $[r,r]_\g=-2e_1\wedge e_2\wedge e_3$ and $\ad_x[r,r]_\g=0$ for all $x\in\g.$ Thus, conditions in Theorem \ref{thm:coboundary} are satisfied and this gives rise to a coboundary purely Hom-Lie bialgebra. More precisely, we have
  $$
  \Delta(e_1)=\frac{1}{a}e_1\wedge e_2,\quad\Delta(e_2)=0,\quad \Delta(e_3)=-ae_2\wedge e_3,
  $$
  and the Hom-Lie algebra on the dual space $\g^*$ is given by
  $$
  [e^1,e^2]_{\g^*}=\frac{1}{a}e^1,\quad[e^1,e^3]_{\g^*}=0, \quad [e^2,e^3]_{\g^*}=-ae^3,
  $$
  where $\{e^1,e^2,e^3\}$ are the dual basis.
  }
\end{ex}

\begin{defi}
Let $\g$ be a regular Hom-Lie algebra and $r\in\wedge^2\g$ satisfying \eqref{eq:conr}. The equation
 \begin{eqnarray}\label{CYBE}
 ~[r,r]_\g=0.
 \end{eqnarray}
is called the {\bf classical Hom-Yang-Baxter equation}.
  A \textbf{triangular purely Hom-Lie bialgebra} is a coboundary purely Hom-Lie bialgebra, in which $r$ is a solution of the  classical Hom-Yang-Baxter equation.
\end{defi}

\begin{rmk}
  Due to the condition \eqref{eq:coboundary}, the classical Hom-Yang-Baxter equation defined above is not the same as existing ones.
\end{rmk}

\begin{ex}\label{ex:tri}{\rm
  Let $\{e_1,e_2\}$ be a basis of a 2-dimensional vector space $\g$ and $\{e^1,e^2\}$ the dual basis of $\g^*$. Define a skew-symmetric bracket operation $[\cdot,\cdot]_\g$ on $\g$ by
  $$
  [e_1,e_2]_\g=e_2.
  $$
  Then $(\g,[\cdot,\cdot]_\g,\phi)$ is a Hom-Lie algebra in which $\phi=\left(\begin{array}{cc}1&1\\0&1\end{array}\right)$. See \cite{CIP} for more details about the classification of 2-dimensional Hom-Lie algebra. Since $\wedge^3\g=0$, any $r\in\wedge^2\g$ is a solution of the classical Hom-Yang-Baxter equation. In particular, let $r=e_1\wedge e_2$. Then the Hom-Lie algebra structure on $\g^*$ is given by
  $$
  [e^1,e^2]_{\g^*}=e^1,\quad (\phi^{-1})^*=\left(\begin{array}{cc}1&0\\-1&1\end{array}\right)
  $$
  Consequently, $(\g,[\cdot,\cdot]_\g,\phi)$ and $(\g^*,[\cdot,\cdot]_{\g^*},(\phi^{-1})^*)$ constitute a triangular purely Hom-Lie bialgebras.
  }
\end{ex}

\begin{cor}\label{rYBE}
 If $r\in\wedge^2\g$  a solution of the classical Hom-Yang-Baxter equation \eqref{CYBE}, then $r^\sharp$ is a homomorphism from the Hom-Lie algebra $(\g^*,[\cdot,\cdot]_{\g^*},(\phi_\g^{-1})^*)$ to the Hom-Lie algebra $(\g,[\cdot,\cdot]_{\g},\phi_\g)$.
\end{cor}

At the end of this section, we construct a solution of the classical Hom-Yang-Baxter equation using  the Hom-$\huaO$-operator introduced in \cite{caisheng}, which is a generalization of an $\huaO$-operator introduced by Kupershmidt in \cite{Kupershmidt2}.
\begin{defi}\label{ooperator}
Let $(\g,[\cdot,\cdot]_\g,\phi)$ be a Hom-Lie algebra and $\rho:\g\longrightarrow\gl(V)$  a representation of $(\g,[\cdot,\cdot]_\g,\phi)$ on $V$ with respect to $\beta\in GL(V)$. A linear map $T:V\rightarrow \g$ is called a Hom-$\huaO$-operator  if $T$ satisfies
    \begin{eqnarray}
    \label{o1}T\circ \beta&=&\phi\circ T;\\
       \label{o2}[Tu,Tv]_{\frkg}&=&T\big(\rho(T(\beta^{-1}(u)))(v)-\rho(T(\beta^{-1}(v)))(u)\big).
       \end{eqnarray}
  \end{defi}
 By Proposition \ref{coobracket}, we have the following example of Hom-$\huaO$-operators.
  \begin{ex}{\rm Let $(\g,[\cdot,\cdot]_\g,\phi)$ be a Hom-Lie algebra. Then $r\in\wedge^2\g$ satisfies the classical Hom-Yang-Baxter equation \eqref{CYBE} and condition \eqref{eq:conr} if and only if $r^\sharp\circ(\phi^{-1})^*$ is a Hom-$\huaO$-operator associated to the coadjoint representation $\ad^\star$.}
\end{ex}

\begin{ex}{\rm
  Let $(\g,[\cdot,\cdot]_\g,\phi)$ be the 2-dimensional Hom-Lie algebra given in Example \ref{ex:tri}. Let $\{e_1,e_2\}$ be a basis of $\g$ and $\{e^1,e^2\}$ the dual basis. Then $\left(\begin{array}{cc}-1&1\\-1&0\end{array}\right):\g^*\longrightarrow\g$ is an Hom-$\huaO$-operator associated to the coadjoint representation $\ad^\star$.
  }
\end{ex}

Next, we consider the semi-direct product Hom-Lie algebra $\g\ltimes_{\rho^\star}V^*$. Any linear map $T: V\rightarrow\g$ can be view as an element $\bar{T}\in\otimes^2(\g\oplus V^*)$ via
$$\bar{T}(\xi+u,\eta+v)=\langle T(u),\eta\rangle,\quad\forall \xi+u,\eta+v\in\g^*\oplus V.$$
Let $\kappa$ be the exchange operator acting on the tensor space, then $r:=\bar{T}-\kappa(\bar{T})$ is skew-symmetric.

\begin{thm}\label{thm:HYB}
Let $(\g,[\cdot,\cdot]_\g,\phi)$ be a Hom-Lie algebra, $\rho:\g\longrightarrow\gl(V)$  a representation of $(\g,[\cdot,\cdot]_\g,\phi)$ on $V$ with respect to $\beta\in GL(V)$ and $T:V\rightarrow\g$  a linear map satisfying $T\circ\beta=\phi\circ T$. Then  $r=\bar{T}-\kappa(\bar{T})$ is a solution of the classical Hom-Yang-Baxter equation in the Hom-Lie algebra $\g\ltimes_{\rho^\star}V^*$ if and only if $T\circ\beta$ is an Hom-$\huaO$-operator.
\end{thm}
\begin{proof}
Let $\{v_1,\cdots,v_n\}$ be a basis of $V$ and $\{v^1,\cdots,v^n\}$ be its dual basis. It is obvious that $\bar{T}$ can be expressed by $\bar{T}=v^i\otimes T(v_i)$. Here the Einstein summation convention is used. Therefore, we can write $r=v^i\wedge T(v_i)$. By direct computations, we have
\begin{eqnarray*}
~[r,r]_{\g\ltimes_{\rho^\star}V^*}
&=&[v^i\wedge T(v_i),v^j\wedge T(v_j)]_{\g\ltimes_{\rho^\star}V^*}\\
&=&(\beta^{-1})^*(v^i)\wedge(\beta^{-1})^*(v^j)\wedge[T(v_i),T(v_j)]_\g\\
&&+(\beta^{-1})^*(v^i)\wedge[T(v_i),v^j]_{\g\ltimes_{\rho^\star}V^*}\wedge\phi(T(v_j))\\
&&-(\beta^{-1})^*(v^j)\wedge[v^i,T(v_j)]_{\g\ltimes_{\rho^\star}V^*}\wedge\phi(T(v_i))\\
&=&\langle(\beta^{-1})^*(v^i),v_m\rangle v^m \wedge\langle(\beta^{-1})^*(v^j),v_n\rangle v^n\wedge[T(v_i),T(v_j)]_\g\\
&&+\langle(\beta^{-1})^*(v^i),v_m\rangle v^m\wedge\langle\rho^\star(T(v_i))(v^j),v_n\rangle v^n\wedge T(\beta(v_j))\\
&&+\langle(\beta^{-1})^*(v^j),v_n\rangle v^n\wedge\langle\rho^\star(T(v_j))(v^i),v_m\rangle v^m\wedge T(\beta(v_i))\\
&=&v^m\wedge v^n\wedge\langle v^i,\beta^{-1}(v_m)\rangle\langle v^j,\beta^{-1}(v_n)\rangle[T(v_i),T(v_j)]_\g\\
&&-v^m\wedge v^n\wedge\langle v^j,\rho\big{(}T(\beta^{-1}(v_i))\big{)}(\beta^{-2}(v_n))\rangle\langle v^i,\beta^{-1}(v_m)\rangle  T(\beta(v_j))\\
&&+v^m\wedge v^n\wedge\langle v^i,\rho\big{(}T(\beta^{-1}(v_j))\big{)}(\beta^{-2}(v_m))\rangle\langle v^j,\beta^{-1}(v_n)\rangle  T(\beta(v_i))\\
&=&v^m\wedge v^n\wedge
\big{(}[T(\beta^{-1}(v_m)),T(\beta^{-1}(v_n))]_\g
-T\beta\big{(}\rho(T(\beta^{-2}(v_m)))(\beta^{-2}(v_n))\big{)}\\
&&+T\beta\big{(}\rho(T(\beta^{-2}(v_n)))(\beta^{-2}(v_m))\big{)}\big{)}\\
&=&v^m\wedge v^n\wedge
\big{(}[T\beta(\beta^{-2}(v_m)),T\beta(\beta^{-2}(v_n))]_\g
-T\beta\big{(}\rho(T\beta(\beta^{-1}\cdot\beta^{-2}(v_m)))(\beta^{-2}(v_n))\big{)}\\
&&+T\beta\big{(}\rho(T\beta(\beta^{-1}\cdot\beta^{-2}(v_n)))(\beta^{-2}(v_m))\big{)}\big{)}.
\end{eqnarray*}

Furthermore, since $T\circ\beta=\phi\circ T$, it is obvious that $T\circ\beta$ satisfies $(T\circ\beta)\circ\beta=\phi\circ(T\circ\beta).$ Thus, $r$ is a solution of the classical Hom-Yang-Baxter equation in the Hom-Lie algebra $\g\ltimes_{\rho^\star}V^*$ if and only if $T\circ\beta$ is a Hom-$\huaO$-operator.
\end{proof}

\end{document}